\documentclass{amsart}
\usepackage{verbatim,amssymb,amsmath,amscd,latexsym,amsbsy,mathrsfs}
\input{xy} \xyoption{all}

\newtheorem{thm}{Theorem} \newtheorem{pro}[thm]{Proposition}
\newtheorem{lm}[thm]{Lemma}
\newtheorem{cor}[thm]{Corollary} 
\theoremstyle{definition}
\newtheorem{df}[thm]{Definition} 
\newtheorem*{notation}{Notation}
\newtheorem*{ex}{Example}

\newtheorem*{clm}{Claim}
\newtheorem{rmk}[thm]{Remark}

\DeclareMathAlphabet{\mathpzc}{OT1}{pzc}{m}{it}
\DeclareMathOperator*{\spec}{Spec}

\DeclareMathOperator*{\Gal}{Gal} \DeclareMathOperator*{\Ind}{\mathrm Ind}
 
\DeclareMathOperator*{\rank}{rank}
\DeclareMathOperator*{\normal}{\trianglelefteq}

\DeclareMathOperator*{\image}{im}
\DeclareMathOperator*{\into}{\hookrightarrow}
 
\newcommand{\RR}{\mathbb{R}} 
\newcommand{\ZZ}{\mathbb{Z}} \newcommand{\Aff}{\mathbb{A}}
\newcommand{\CC}{\mathbb{C}} 
\newcommand{\PP}{\mathbb{P}}

\newcommand{\cO}{\mathcal{O}} 
\newcommand{\cE}{\mathcal{E}} 
 
 \newcommand{\cK}{\mathcal{K}}

\newcommand{\scrY}{\mathcal{Y}}

 \newcommand{\onto}{\twoheadrightarrow}

\newcommand{\Z}{\mathbb{Z}}

\newcommand{\caa}{\mathcal{A}}\newcommand{\cab}{\mathcal{B}}\newcommand{\cac}{\mathcal{C}}\newcommand{\cak}{\mathcal{K}}\newcommand{\cam}{\mathcal{M}}\newcommand{\cao}{\mathcal{O}}\newcommand{\cas}{\mathcal{S}}
\newcommand{\frF}{\mathfrak{F}}\newcommand{\frG}{\mathfrak{G}}\newcommand{\frp}{\mathfrak{p}}
\begin{document}
\title{Embedding problems for open subgroups of the fundamental group}
\author{Manish Kumar} \address{
  Statistics and Mathematics Unit\\
  Indian Statistical Institute, \\
  Bangalore, India-560059
    } \email{manish@isibang.ac.in} 
\begin{abstract}
  Let $C$ be a smooth irreducible affine curve $C$ over an
  algebraically closed field of positive characteristic and let
  $\pi_1(C)$ be its fundamental group. We study various embedding problems 
  for $\pi_1(C)$ and its subgroups.
\end{abstract}
\date{} \maketitle

\section{Introduction}
Let $C$ be a smooth irreducible affine curve over an algebraically
closed field $k$ of characteristic $p>0$. The existence of wild ramification causes the structure of the \'etale fundamental group $\pi_1(C)$ to be complicated. There has been some attempts to understand this group. One step towards this goal was Raynaud's \cite{Ray} and Harbater's \cite{ha94} proof of the Abhyankar's conjecture which states that a finite group $G$ is a quotient of $\pi_1(C)$ if and only if the maximal prime-to-$p$ quotient $G/p(G)$ of $G$ is generated by $2g+r-1$ elements where $g$ is the genus of the smooth completion of $C$ and $r$ is the number of points in the boundary.

Though this gives a complete description of the finite quotients of $\pi_1(C)$, it does not say how these groups fit together in the inverse system for $\pi_1(C)$. A possible way to understand this is by analyzing which finite embedding problems for $\pi_1(C)$ have a solution (see the ``Notation'' subsection at the end of this introduction for definition). One crucial result in this direction is by Pop \cite{pop} (and independently proved by Harbater as well \cite[Theorem 5.3.4]{ha03}) which says that given an embedding problem  $\cE=(\alpha:\Pi \to G, \phi:\Gamma \to G)$ where $H = \ker(\phi)$ is a quasi-$p$ group, there is a proper solution $\lambda: \Pi \twoheadrightarrow \Gamma$ to $\cE$. This clearly strengthens Raynaud's and Harbater's result.
When $H$ is not a quasi-$p$ group then it is clear from the Abhyankar's conjecture that the above embedding problem may not have a solution. But there are finite index open subgroups of $\pi_1(C)$ for which these embedding problems have a solution.

\begin{df}
  Given an embedding problem $\cE=(\alpha:\pi_1(C) \to G, \phi:\Gamma \to G)$ and a finite index subgroup
  $\Pi$ of $\pi_1(C)$, we say the embedding problem \emph{restricts} to $\Pi$
  if $\alpha(\Pi)=G$.  Moreover if the restricted embedding problem
  has a solution then we say $\Pi$ is \emph{effective} for the
  embedding problem $\cE$.
\end{df}

In \cite{HS}, it was shown that given any finite embedding problem for $\pi_1(C)$ 
there exist a finite index effective subgroup for the embedding problem. In 
\cite[Theorem 1.3]{BK} it was shown that it is even possible to find an index $p$ 
subgroup of $\pi_1(C)$ which is effective for the embedding problem. 
Our objective is to find some necessary and some sufficient conditions for a 
subgroup of $\pi_1(C)$ to be an effective subgroup for a given embedding problem.

Suppose $\Pi$ is a finite index subgroup of $\pi_1(C)$. 
This corresponds to a cover $D \to C$.
Let $Z \to X$ be the corresponding morphism between their smooth completions and let $n_D$ denote $2g_Z+r_D-1$. 
A necessary condition for $\Pi$ to be effective for the embedding problem $\cE=(\alpha:\pi_1(C) \to G, \phi:\Gamma \to G)$ is that $\cE$ restricts to $\Pi$ and the rank of $\Gamma/p(\Gamma)$ is at most $n_D$. Indeed $\Pi$ is effective implies there is a surjection from $\pi_1(D)=\Pi\onto \Gamma$. Hence there is a surjection from prime-to-$p$ part of $\pi_1(D)$ to the maximal prime-to-$p$ quotient  of $\Gamma$, denoted $\Gamma/p(\Gamma)$. So $\Gamma/p(\Gamma)$ is generated by $n_D$ elements.  Also, \cite[Theorem 5]{HS} can be rephrased in terms of sufficient conditions for $\Pi = \pi_1(D)$ to be an effective subgroup of $\pi_1(C)$.

\begin{pro}\label{relativerank}
  Let $C\subset X$, $\Pi$, $\cE$, and $D \subset Z$ be as above. Let $\Psi_X:V_X\to X$ be the
  $G$-cover corresponding to the embedding problem $\cE$ and
  assume $\cE$ restricts to $\Pi$. Let $\tilde \Psi:V_X \times_X Z \to Z$ be the pull-back of $\Psi_X$. 
  Then $\Pi$ is effective if the number of points in $Z \setminus D$ where the morphism $\tilde \Psi$ is not branched, is at least the relative rank of $\ker(\phi:\Gamma \to G)$ in $\Gamma$.  
\end{pro}

The subset of $Z \setminus D$ for which the morphism $\Psi$ is unramified can be thought of as points available for branching in any cover dominating $\Psi$.  In the theorem above, having sufficiently many available branch points allows the embedding problem to be solved.  A natural question is can the condition on the number of such potential branch points in $Z \setminus D$ be relaxed?  For example, could it be replaced by a condition on the genus of $Z$ or $n_D$? It is also worth noting that in the above proposition the degree of the cover $Z\to X$ must be large to ensure that the number of points in $Z\setminus D$ where the morphism $\tilde \Psi$ is not branched is sufficiently large. Hence the index of $\Pi$ in $\pi_1(C)$ is also large. 

The two main results of this paper, Theorem \ref{main.result-affine.line} and  Proposition  \ref{degeneration.highergenus},  use genus and branch points to obtain effective subgroups.  Let $H$ be the kernel of the homomorphism $\phi: \Gamma \to G$ from the embedding problem $\cE = (\alpha:\pi_1(C) \to G, \phi:\Gamma \to G)$, and let $H/p(H)$ be the maximal prime to $p$ quotient of $H$.   The main tool in this paper is Theorem \ref{EPsoln_eff_degeneration}  which investigates the relationship between solving embedding problems for $\pi_1(C)$ and the relative rank of $H/p(H)$ in $\Gamma/p(H)$. Roughly speaking, it shows that if the $G$-Galois cover $\Psi_X:V_X \to X$ corresponding to $\alpha$ has a deformation which is sufficiently degenerate (in terms of having many components that are trivial $G$ covers), then there exists a proper solution to $\cE$. 
Theorem \ref{main.result-affine.line} restricts to the case that $C$ is the affine line and $\Pi$ has index $p$ in $\pi_1(\Aff^1_k)$.  It uses degenerations and  Theorem \ref{EPsoln_eff_degeneration} to show that if the curve $Z$ in the $\Z/p\Z$-cover $Z \to \PP^1_k$ corresponding to $\Pi \subset \pi_1(\Aff^1_k)$ has genus $g_Z$ greater than the relative rank of $(H/pH)$ in $\Gamma/p(H)$ and a technical condition which holds for most values of $g_Z$ (see Corollary \ref{cor1:main.result-affine.line}, Corollary \ref{cor2:main.result-affine.line} and Remark \ref{rmk:main.result-affine.line}) then $\Pi$ is effective for $\cE$.  
 In particular these results provide sufficient conditions for subgroups of $\pi_1(C)$ to be effective. In Section \ref{existdegens} some of the results proved for affine line case are generalized to general curves though the conclusions obtained are slightly weaker (Proposition \ref{degeneration.highergenus} and Corollary \ref{general.case}).

The techniques involved in proving these results include formal patching and deformations within families of covers.  Each result requires the construction of a Galois cover of a degenerate curve.  Then a formal patching argument is used to obtain a cover of smooth curves over a complete local ring.  Finally, a deformation argument similar to \cite[Proposition 4]{HS}  is used to ensure that the original embedding problem $\cE$ restricts the subgroup $\Pi$ of $\pi_1(C)$ that is obtained in the construction.  The paper is organized as follows: 
At the end of the Introduction there is a list of notation.   Section \ref{GroupTheory} looks at some of the group theoretic properties and examples of effective groups.   
Section \ref{Formal Patching}  defines deformations and degenerations of covers and uses a formal patching result of Harbater to solve the embedding problems in the presence of a sufficiently degenerate deformation of the original cover.  
Section \ref{LATrick} proves a globalization and specialization result using the Lefschetz type principle and Abhyankar's lemma.  
Section \ref{degenerates} applies the formal patching and Lefschetz-Abhyankar result to obtain Theorem \ref{EPsoln_eff_degeneration}.
Finally in Section \ref{Affine Line}, Theorem \ref{main.result-affine.line} is proved, and in Section \ref{existdegens}, Proposition \ref{degeneration.highergenus} and Corollary \ref{general.case} is proved.

\medskip

\begin{notation}
If $G$ is a finite group and $p$ is a prime number, let $p(G)$ denote the subgroup of $G$ generated by its $p$-subgroups.  
This is a characteristic subgroup of $G$, and $G/p(G)$ is the maximal prime-to-$p$ quotient of $G$. 
A finite group $G$ is called {\it quasi-$p$} if $G = p(G)$.

Let $\Gamma$ be any finite group  and let $H$ be a subgroup of $\Gamma$.
A subset $S \subset H$ will be called a {\it relative generating set} for $H$ in $\Gamma$ if for every subset $T \subset \Gamma$ such that $H \cup T$  generates $\Gamma$, the subset $S \cup T$ also generates $\Gamma$.
We define {\it the relative rank of $H$ in $\Gamma$} to be the smallest non-negative integer $\mu:=\rank_\Gamma(H)$ such that there is a relative generating set for $H$ in $\Gamma$ consisting of $\mu$ elements.  
Every generating set for $H$ is a relative generating set, so $0 \leq \rank_\Gamma(H) \leq \rank(H).$
Also, $\rank_\Gamma(H) = \rank(H)$ if $H$ is trivial or $H = \Gamma$, while $\rank_\Gamma(H) = 0$ if and only if $H$ is contained in the Frattini subgroup $\Phi(\Gamma)$ of $\Gamma$ \cite[p.~122]{ha99}.  

A {\it finite embedding problem} $\cE$ for a group $\Pi$ is a pair of surjections $(\alpha:\Pi \to G, \phi:\Gamma \to G)$, where $\Gamma$ and $G$ are finite groups.  If $H = \ker(\phi)$, the embedding problem $\cE$ can be summarized by

%\begin{equation}
 \centerline{ \xymatrix{
    &          &               &\Pi \ar[d]^{\alpha} \ar @{-->} [dl]\\
    1\ar[r] & H \ar[r] & \Gamma \ar[r]^{\phi} & G \ar[r]\ar[d] & 1\\
    & & & 1 }}
%\end{equation}

\smallskip \noindent A {\it weak solution} to $\mathcal{E}$ is a homomorphism $\gamma: \Pi \to \Gamma$ such that $\phi \circ \gamma = \alpha$.  We call $\gamma$
a {\it proper solution} to $\mathcal E$ if in addition it is surjective.

\begin{rmk}\label{split}Let $\Pi$ be a profinite group and let $\cE = (\alpha:\Pi \to G, \phi:\Gamma \to G)$ be a finite embedding problem for $\Pi$.  Suppose that the epimorphism $\alpha:\Pi \to G$ factors as $r\alpha'$, where $\alpha':\Pi \to G'$ and $r:G'\to G$ are 
epimorphisms, for some finite group $G'$.  We consider the {\it induced embedding problem} $\cE_{\alpha'} = (\alpha':\Pi \to G', \phi':\Gamma' \to G')$ by taking $\Gamma' = \Gamma \times_G G'$ and letting $\phi':\Gamma' \to G'$ be the second projection map.  Here 
$\phi'$ is surjective because $\phi$ is; and so $\cE_{\alpha'}$ is a finite embedding problem.  Here $\cE$ and $\cE_{\alpha'}$ have isomorphic kernels; indeed $\ker(\phi') = \ker(\phi) \times 1 \subset \Gamma \times_G G' = \Gamma'$.  Note that the first projection map $q:\Gamma' = \Gamma \times_G G' \to \Gamma$ is surjective since $r:G' \to G$ is surjective; and $\phi q = r \phi'$.   

In this situation, every proper solution $\lambda':\Pi\to \Gamma'$ of $\cE_{\alpha'}$ induces a proper solution $\lambda := q\lambda':\Pi \to \Gamma$ of $\cE$; viz.\ $\phi\lambda = \phi q \lambda' = r \phi' \lambda' = r\alpha' = \alpha$, and 
$\lambda$ is surjective because $q$ and $\lambda'$ are.  So we obtain a map ${\rm PS}(\cE_{\alpha'}) \to {\rm PS}(\cE)$, where $\rm PS$ denotes the set of proper solutions to the embedding problem.  
\end{rmk}
In this paper we consider curves over an algebraically closed field $k$ of characteristic $p>0$.  
A {\it cover} of $k$-curves is a morphism $\Phi:D \to C$ of smooth connected $k$-curves that is finite and generically separable.  
If $\Phi:D \to C$ is a cover, its \emph{Galois group} $\mathrm{Gal}(D/C)$ is the group of $k$-automorphisms $\sigma$ of $D$ satisfying $\Phi \circ \sigma = \Phi$.  
If $G$ is a finite group, then a $G$-\emph{Galois cover} is a cover $\Phi: D \to C$ \emph{together} with an inclusion $\rho: G \hookrightarrow \mathrm{Gal}(D/C)$ such that $G$ acts simply transitively on a generic geometric fiber of $\Phi:D \to C$.  If we fix a geometric point of $C$ to be a base point, then the pointed $G$-Galois \'etale covers of $C$ correspond bijectively to the surjections $\alpha: \pi_1(C) \to G$, where $\pi_1(C)$ is the algebraic fundamental group of $C$ with the chosen geometric point as the base point.  
The proper solutions to an embedding problem $\cE = (\alpha:\pi_1(C) \to G, \phi:\Gamma \to G)$ for $\pi_1(C)$ then are in bijection to the pointed $\Gamma$-Galois covers $E \to C$ that dominate the pointed $G$-Galois cover $\Phi: D \to C$ corresponding to $\alpha$.
In the case that $X$ is the smooth completion of the affine $k$-curve $C$, denote by $g_X$ the genus of $X$, and define $r_C = \#(X - C)$ and $n_C = 2g_X+r_C-1$. 
\end{notation}

\section*{Acknowledgments}
The author thanks David Harbater and Kate Stevenson for a number of useful discussions which led to some of the ideas in the paper. There were also many suggestions from them which improved the presentation of this paper.  

\section{Group theory results and examples of effective subgroups}\label{GroupTheory}

In this section we analyze when a subgroup of an effective subgroup for an embedding problem is effective using some group theory and Galois theory.

We start with an embedding problem $\cE$\eqref{EP} for $\Pi$.
\begin{equation}\label{EP}
  \xymatrix{
    &          &               &\Pi \ar[d]^{\alpha} \ar @{-->} [dl]\\
    1\ar[r] & H \ar[r] & \Gamma \ar[r]^{\phi} & G \ar[r]\ar[d] & 1\\
    & & & 1 }
\end{equation}

The following remark is easy to see.
\begin{rmk} \label{rmk:subgroup} Suppose $\Pi_1$ is an effective subgroup for the embedding
  problem $\cE$\eqref{EP} and $\Pi_2$ is a finite index subgroup of $\Pi_1$. Suppose for 
  some solution $\psi$ of the embedding problem $\cE$\eqref{EP} restricted to $\Pi_1$,
  the index of $\ker(\psi)\cap \Pi_2$ in $\Pi_2$ is $|\Gamma|$ then $\Pi_2$
  is also an effective subgroup for $\cE$\eqref{EP}.
  Note that $\ker(\psi|_{\Pi_2})=\ker(\psi)\cap \Pi_2$ has index $|\Gamma|$ in $\Pi_2$. Hence $\psi|_{\Pi_2}$ indeed surjects onto $\Gamma$. 
\end{rmk}

\begin{cor} \label{subgroup} Suppose $\Pi_1$ is an effective subgroup for the embedding
  problem $\cE$\eqref{EP} and $\Pi_2$ is a finite index subgroup of $\Pi_1$
  such that the embedding problem $\cE$\eqref{EP} restricts to $\Pi_2$ and
  $[\Pi_1:\Pi_2]$ is coprime to $|H|$ then $\Pi_2$ is also effective for
  $\cE$\eqref{EP}.
\end{cor}

\begin{proof}
Let $\psi$ be a solution to the embedding problem $\cE$\eqref{EP} restricted to $\Pi_1$. 
Note that $\ker(\psi)\subset \ker(\alpha|_{\Pi_1})$ and $[\ker(\alpha|_{\Pi_1}):\ker(\psi)]=|H|$.
Since $[\Pi_1:\Pi_2]$ is coprime to $|H|$, $[\ker(\alpha|_{\Pi_1}):\ker(\alpha|_{\Pi_2})]$ is also
coprime to $|H|$. So the index $[\ker(\alpha|_{\Pi_2}):\ker(\psi)\cap\ker(\alpha|_{\Pi_2})]=|H|$.
The embedding problem restricts to $\Pi_2$ means that $[\Pi_2:\ker(\alpha|_{\Pi_2})]=|G|$.
So we obtain that $[\Pi_2: \ker(\psi)\cap\ker(\alpha|_{\Pi_2})]=|G||H|=|\Gamma|$. Also note
that $\ker(\psi)\cap\ker(\alpha|_{\Pi_2})=\ker(\psi)\cap \ker(\alpha|_{\Pi_1})\cap \Pi_2
=\ker(\psi)\cap \Pi_2$. Hence the result is obtained from the previous remark.
\end{proof}

In the above corollary the hypothesis guaranteed that if $\Pi_1$
is an effective subgroup for the given embedding problem and $\psi$
is a solution of the embedding problem restricted to $\Pi_1$ then 
$\psi|_{\Pi_2}$ is a solution to the embedding problem restricted to 
$\Pi_2$. Hence $\Pi_2$ is also an effective subgroup. But this does not hold 
unconditionally as the following examples show.

\begin{ex}
Let $\Pi$ be the absolute Galois group of the reals, let $\Gamma=H=\ZZ/2\ZZ$, 
and let $G=1$.  Then the given embedding problem has a proper solution 
(the complex numbers). So $\Pi_1=\Pi$ itself is effective. But if we pull back from $\RR$ to $\CC$ 
(corresponding to taking the trivial subgroup $\Pi_2$ of $\Pi$) then the 
embedding problem (which has trivial cokernel) restricts to $\Pi_2$. But it no longer has a proper solution. Note here $|H|=[\Pi:\Pi_2]=2$.
\end{ex}

On the contrary, in geometric setting, even if the given solution does not 
pull back to a proper solution, there might be some other proper 
solution over the pullback.

\begin{ex}
Let $C$ be the affine $x$-line minus 0 in characteristic 0, let $\Pi = \pi_1(C)$, 
$\Gamma=\ZZ/2\ZZ \times \ZZ/3\ZZ$, $H=\ZZ/3\ZZ$, and $ G=\ZZ/2\ZZ$.  Then there is a G-cover $C_1$ 
of $C$ given by $y^2=x$.  Over that there's a proper solution $D$ to 
the $\cE$\eqref{EP}, given by $z^3=y$. So again $\Pi_1=\Pi$ is an effective subgroup of $\Pi$ for $\cE$\eqref{EP}. Here $D$ is the fiber product of $C_1$ with 
the curve $C_2$ given by $w^3=x$ (where $w=z^2$, and $z=y/w$).  Now pull 
back everything by the degree 3 cover $C_2$ of $C$.  This is linearly 
disjoint from $C_1$, so $\cE$\eqref{EP} restricts to the subgroup $\pi_1(C_2)$ of $\Pi$.  But the 
degree is not relatively prime to $|H|$, the cover $C_2$ is not linearly 
disjoint from $D$, and the solution to the given $\cE$\eqref{EP} does not restrict to a proper solution to $\cE$\eqref{EP} restricted to $\pi_1(C_2)$. But $\cE$\eqref{EP} restricted to $\pi_1(C_2)$ does have a proper solution, given by $v^3=z$.
\end{ex}

\section{Formal Patching Results}\label{Formal Patching}
In this section we develop some formal patching results which are used in later sections to find solutions to various embedding problems. Proposition \ref{patchKumar} is the main result of this section and one of the main technical results of this paper.

\begin{notation} Given a scheme $X$, denote by $\cam(X)$ the category of coherent sheaves of $\cao_X$-modules, $\caa\cam(X)$ the category of coherent sheaves of $\cao_X$-algebras and $\cas\cam(X)$ the subcategory of $\cam(X)$ for which the sheaves of algebras are generically separable and locally free. Given a finite group $G$ denote by $G\cam(X)$ the category of generically separable coherent locally free sheaves of $\cao_X$-algebras $S$ together with a $G$-action which is transitive on the geometric generic fibers of $\spec_{\cao_X}(S)\to X$. Given categories $\caa,\cab,\cac$ and functors $\frF:\caa\to\cac$ and $\frG:\cab\to\cac$, denote by $\caa\times_{\cac}\cab$ the associated fiber category.\end{notation}  
The following result is due to Harbater \cite[Theorem 3.2.12]{ha03}.

\begin{thm}\label{genpatch} 
Let $(A,\frp)$ be a complete local ring and $T$ a proper $A$-scheme. Let $\{\tau_1,\cdots,\tau_N\}$ be a finite set of closed points of $T$ and $U=T\setminus\{\tau_1,\cdots,\tau_N\}$. Denote by $\hat{\cao}_{T,\tau_i}$ the completion of the local ring $\cao_{T,\tau_i}$ and let ${T}_i=\spec(\hat{\cao}_{T,\tau_i})$. Let $U^*$ be the $\frp$-adic completion of $U$ and $\cak_i$ the $\frp$-adic completion of $T_i\setminus\{\tau_i\}$. Then the base change functor
$$\cam(T)\to\cam(U^*)\times_{\cam\left(\bigcup_{i=1}^N\cak_i\right)}\cam\left(\bigcup_{i=1}^N{T}_i \right)$$
is an equivalence of categories. The same remains true if we replace $\cam$ by $\caa\cam$, $\cas\cam$ or $G\cam$ for a fixed finite group $G$. \end{thm}

Next we state a lemma useful for putting covers into a situation where the theorem above holds.  Its proof is just an application of a generalization of the Noether Normalization Lemma.

\begin{lm}\label{Noether}
Let $C$ be an affine $k$-curve and $X$ be its smooth projective completion.
Then there exists a finite morphism $\Theta_{X}:X\to\PP^1_x$ such that
$\Theta_{X}$ is {\'e}tale at $\Theta_{X}^{-1}(\{x = 0\})$ and such that $\Theta_{X}^{-1}(\{x = \infty\}) = X \setminus C$.
\end{lm}

\begin{proof}
By a stronger version of Noether
normalization (\emph{cf.} \cite[Corollary 16.18]{Eis}), there exist finite proper generically separable
$k$-morphism from $C \to \Aff^1_x$  where $\Aff^1_x$ is the affine $k$-lines with local coordinate $x$. 
The branch locus of this morphism is of codimension $1$, and hence it is {\'e}tale away from finitely many points. 
By translation we may assume that $x=0$ is not a branch point of $C \to \Aff^1_x$.
This morphism extends to a finite proper morphism $\Theta_{X}: X \to \PP^1_x$. 
Note that $\Theta_{X}$ is {\'e}tale at $\Theta_{X}^{-1}(\{x = 0\})$ and that $\Theta_{X}^{-1}(\{x = \infty\}) = X \setminus{C}$.
\end{proof}

The patching results will be applied to covers of reducible curves where there are sufficiently many components of the base over which the cover is trivial.  First, we need some terminology.  
\begin{df}
Let $\Phi:V\to X$ be a $G$-cover of smooth irreducible projective curves over $k$.
Assume $X$ has $r$ marked points and $\Phi$ is \'etale away from these $r$ points.
We say that a $G$-cover of connected projective curves $\Phi':V'\to X'$ with $r$ marked
points on $X'$ is a \emph{deformation} of $V\to X$ if there exist a smooth irreducible 
$k$-scheme $S$, a cover of $S$-curves $\Phi_S:V_S \to X_S$, $r$ sections 
$p_1,\ldots, p_r:S \to X_S$ and closed points $s$ and $s'$ in $S$ such that the following three conditions hold.
\begin{enumerate}
 \item $\Phi_S$ induces $\Phi$ and $\Phi'$ at $s$ and $s'$ respectively. 
 \item $\{p_1(s),\ldots, p_r(s)\}$ and $\{p_1(s'),\ldots, p_r(s')\}$ are the marked 
  points of $X$ and $X'$ respectively.
 \item $\Phi_S$ is a $G$-cover \'etale away from the sections $p_1,\ldots, p_r$.
\end{enumerate}
\end{df}

\begin{notation} Given a $\Phi:V\to X$ and a deformation $\Phi':V'\to X'$, we will call $(S; \Phi_S:V_S \to X_S; p_1, \ldots, p_r; s; s')$ the associated data of the deformation.
\end{notation}

\begin{df}
The deformation $V'\to X'$ will be called a \emph{SNC deformation} if all irreducible 
components of $X'$ are smooth and they intersect transversely. Moreover, it will be 
called a \emph{degeneration} if it is SNC and for some irreducible component $X_1$ of $X'$, the restriction of 
the $G$-cover to $X_1$ is induced from a trivial cover, 
i.e. $V'\times_{X'} X_1 \cong \Ind_{e}^G X_1$ as $G$-covers of $X_1$. In this situation $X_1$ will be 
called a trivial component of $\Phi'$.
\end{df}

\begin{rmk}\label{curve} Let $\Phi: V_X \to X$ be a $G$-cover.
Suppose $\Phi$ has a deformation $\Phi':V_X'\to X'$ with associated data $(S; \Phi_S:V_S \to X_S; p_1, \ldots, p_r; s; s')$.
By taking a smooth irreducible $k$-curve in $S$ passing through $s$ and $s'$ and 
pulling back $\Phi_S:V_S\to X_S$ to this curve, we may assume $S$ is a smooth $k$-curve. In other words, if $\Phi'$ is a deformation of $\Phi$ then we may assume $\Phi$ can be deformed to $\Phi'$ along a smooth curve.
\end{rmk}

\begin{ex}
Let $\Phi:V_X\to X$ be a $G$-cover of irreducible smooth projective curves.
Let $\tau$ be a closed point in $X$. Let $S=\Aff^1_t$ and $X_S$ be the the blowup of 
$(\tau,t=0)$ in $X\times S$ and $X'$ be the total transform of the zero locus of
$t=0$ in $X\times S$. Note that $X'$ has two irreducible components, a copy 
of $X$ and the exceptional divisor isomorphic to $\PP^1$ intersecting at $\tau$. 
One can obtain a $G$-cover $\Phi_S:V_{X_S}\to X_S$ obtained by pullback along 
$X_S \to X\times S\to X$. The fiber over $t=0$ induces a $G$-cover 
$\Phi':V_{X'}\to X'$. The exceptional divisor $\PP^1$ is the trivial component of
$\Phi'$. 
\end{ex}

\begin{lm}\label{smoothfibers}
 Let $\Phi:V\to S$ be a flat family of reduced projective irreducible curves in which $S$ is a smooth connected variety and $V$ is a normal variety. Suppose for every point $s\in S$ the normalization of the fiber $V_s$ has the same genus. Then for every closed point $s\in S$ the fiber $V_s$ is smooth. 
\end{lm}

\begin{proof}
 Let $\eta$ be the generic point of $S$. Note that $V_{\eta}$ being a localization of $V$ is normal. Hence there exist a nonempty open subset $U$ of $S$ such that for all closed point $s\in U$, $V_s$ is a normal $k$-curve and hence smooth. Let $g$ be the genus of such a curve. Suppose there exist $s'\in S$ a closed point such that $V_{s'}$ is singular. Then the arithmetic genus of $V_{s'}$ is $g$, since $\Phi$ is a flat family. But $V_{s'}$ is singular so the geometric $p_g(V_{s'})<g$. But $p_g(V_{s'})$ is same as the genus of the normalization of $V_{s'}$. This contradicts the hypothesis that the normalization of every fiber has the same genus.  
\end{proof}

\begin{pro}\label{patchKumar}
Let $\Gamma$ be a finite group. Let $G$ be a subgroup of $\Gamma$ and let $H_1,\ldots, H_m$ be subgroups of $\Gamma$ of order prime-to-$p$. Assume that $G, H_1,\ldots, H_m$ generate $\Gamma$.
Let $C$ be an affine $k$ curve with smooth completion $X$. Let $\Phi: V_X\to X$ be a $G$-Galois cover of $X$ \'etale over $C$.
Let $B_X = X-C$, $r=\#(B_X)$, and consider the set $B_X$ to be $r$ marked points on $X$.
Suppose $\Phi$ has a degeneration $\Phi':V_X'\to X'$ with an associated data $(S; \Phi_S:V_S \to X_S; p_1, \ldots, p_r; s; s')$ such that $S$ is a smooth curve and $X_1,\ldots, X_m$ are the trivial components of $X'$. Let $B$ be the union of the images of the sections $p_1,\ldots, p_r:S \to X_S$ and let $r_i$ be the number of smooth marked points of $X'$ lying on $X_i$   for $i=1,\ldots,m$. Further assume that for each $1\le i \le m$, there exist an $H_i$-cover $\Phi_{X_i}:W_{X_i}\to X_i$ \'etale away from these $r_i$ points. Let $T=X_S\times_S \hat S_{s'}$ where $S_{s'}$ is the completion of $S$ at $s'$ and $\Phi_T:V_T\to T$ be the pullback of $\Phi_S$ to $T$. 
Then there exists a $\Gamma$-cover $\Psi:W\to T$ such that $\Psi$ is \'etale away from $B\times_{X_S} T$. Moreover, if $\Gamma=H\rtimes G$ and $H_1,\ldots, H_m \le H$ then $\Psi$ dominates the $G$-cover $\Phi_T$.
\end{pro}

\begin{proof}
By Remark \ref{curve} we may assume that there exist an associated data $(S; \Phi_S:V_S \to X_S; p_1, \ldots, p_r; s; s')$ for the degeneration of $\Phi$ to $\Phi'$ such that $S$ is a smooth curve. Let $t$ be the local coordinate of $S$ at $s'$.
Then  $\hat S_{s'}=\spec(k[[t]])$. 
By construction, the closed fiber of $T$ is $X'$. Let $X_0$ be the closure of 
$X'\setminus \cup_{i=1}^m X_i$ in $X'$. So $X_0$ is made up of the 
nontrivial components of $\Phi'$.
Since $\Phi'$ is a degeneration of $\Phi$ it is a SNC deformation, so all irreducible components of $X'$ are smooth and they intersect transversely. 
Let $\tau_1,\ldots, \tau_N$ be the closed points of $T$ where $X_i$ and $X_j$ 
intersect for some $0\le i \ne j \le m$. 
Let $X_i ^o= X_i\setminus \left(\{\tau_1,\ldots,\tau_N\} \cap X_i\right)$ for $0\le i \le m$. 
Note that $\tau_1,\ldots, \tau_N$ will be used to denote the points
of $T$, $X'$ and various components of $X'$ as well, 
but this should not lead to any confusion.

Let $T^o=T\setminus \{\tau_1,\ldots,\tau_N\}$ and $\tilde T^o$ be the formal scheme obtained
by the completion of $T^o$ along the closed fiber $(t=0)$ (i.e. the $(t)$-adic completion).
Let $T_i ^o = X_i ^o\times_k \spec(k[[t]])$ and $\tilde T_i^o$ 
be the $(t)$-adic completion of $T_i ^o$ (i.e. along $X_i^o$) for $i=0,\ldots, m$.
Since the closed fiber of $T^o$ is the disjoint union of $X_0 ^o, X_1 ^o, \ldots, X_m^o$,
$$\tilde T^o = \tilde T_0^o \cup \tilde T_1^o \cup\ldots \cup T_m^o$$

By base change of $\Phi_T$ to $\tilde T^o$ we obtain a $G$-cover 
$\Phi_{\tilde T^o}:\tilde V_{T^o}\to \tilde T^o$ and hence a $G$-cover of the component $\tilde T_0^o$ which we will denote by
$\Phi_{T_0}:\tilde V_0\to \tilde T_0^o$. Note that $\Phi_T$ restricted to the closed 
fiber is the $G$-cover $\Phi':V_X'\to X'$. Since $X_i$, for $i=1,\ldots, m$,  
are the trivial components of 
the cover $\Phi':V_X'\to X'$, the $G$-covers of the components $\tilde T_i^o$ are induced from
the trivial cover. Let $\Phi_{T_i}:\tilde V_i \to \tilde T_i^o$ be the $H_i$-cover obtained
by pulling back $\Phi_{X_i}:W_{X_i}\to X_i$ along the composition of morphisms
$\tilde T_i^o \to T_i ^o\to X_i ^o \to X_i$, for $i=1,\ldots, m$.
Let $\Phi^o:\Ind^{\Gamma}_G \tilde V_0 \cup \Ind^{\Gamma}_{H_1} \tilde V_1 \cup
\ldots \cup \Ind^{\Gamma}_{H_m} \tilde V_m \to \tilde T^o$
be the $\Gamma$-cover of $\tilde T^o$ obtained from $\Phi_{T_i}$ for $i=0,\ldots, m$.

Let $\hat T_{\tau_j}=\spec(\hat\cO_{T,\tau_j})$ be the formal neighbourhood of $\tau_j$ in $T$
and $\hat \cK_j$ be the $(t)$-adic completion of $\hat T_{\tau_j}\setminus \{\tau_j\}$ for 
$1\le j \le N$, i.e. $\hat \cK_j$ is the $(t)$-adic completion of the punctured 
formal neighbourhood of $\tau_j$ in $T$. We have a natural morphism $\hat \cK_j \to \tilde T^o$.
Note that $\Phi':V_X'\to X'$ is \'etale at $\tau_1,\ldots, \tau_N \in X'$, since 
$\tau_1,\ldots,\tau_N$ lie in $X_i$ for some $1\le i\le m$ and over 
$X_i$, $\Phi'$ is induced from a trivial cover. 
Since $\Phi'$ is the restriction of $\Phi_T$ to the closed fiber, 
$\Phi_T$ is \'etale over the points $\tau_1,\ldots,\tau_N$ in $T$. Also note that
$\Phi_{T_i}$ is the pull back of the $H$-cover $\Phi_{X_i}:W_{X_i}\to X_i$ which is 
\'etale over $\tau_1,\ldots,\tau_N$. Hence the pull back of $\Phi^o$ to $\hat \cK_j$ is the
$\Gamma$-cover of $\hat \cK_j$ induced from the trivial cover for $1\le j \le N$. 

Apply Theorem \ref{genpatch} to obtain a $\Gamma$-cover $\Psi:W\to T$ which induces the $\Gamma$-cover 
$\Phi^o$ of $\tilde T^o$ and trivial $\Gamma$-cover over $\cup_{j=1}^N\hat T_{\tau_j}$.
The cover $W$ is connected because $\Gamma$ is generated by 
$G, H_1,\ldots, H_m$ and $\Gamma$ acts transitively on 
$\Ind^{\Gamma}_G \tilde V_0 \cup \Ind^{\Gamma}_{H_1} \tilde V_1 \cup
\ldots \cup \Ind^{\Gamma}_{H_m} \tilde V_m$.

Recall that $B$ is the union of the images of the sections $p_1,\ldots, p_r:S \to X_S$. 
The branch locus of $\Psi$ is clearly contained in $T^o$ because in the formal neighbourhood 
of $\tau_i$'s, $\Psi$ restricts to the $\Gamma$-cover induced from the trivial cover. 
Since the pullback of $\Psi$ to $\tilde T^o$ is $\Phi^o$, the branch locus of 
$\Phi^o$ maps to the branch locus of $\Psi$ under the morphism $\tilde T^o \to T^o$. 
Note that for $i=1, \ldots, m$, $\Phi_{T_i}$ is \'etale away from 
$B_i\times_{X_i^o} \tilde T_i^o$ in $\tilde T_i^o$ where $B_i$ is the set of 
$r_i$ smooth marked points of $X'$ lying on $X_i$. Moreover, the image of 
$B_i\times_{X_i^o} \tilde T_i^o$ under the morphism $\tilde T_i^o\to T$ is 
contained in $B\times_{X_S} T$.

Note that $\Phi_{T_0}:\tilde V_0 \to \tilde T_0$ is the pullback of 
the $G$-cover $\Phi_S:V_S\to X_S$ under the morphism 
$\tilde T_0\to \tilde T^o\to T^o \to T\to X_S$.
Also the branch locus of $\Phi_S$ is contained in $B$. 
So combining all these we see that $\Psi$ is \'etale away from $B\times_{X_S} T$.

Finally if $\Gamma=H\rtimes G$ then $H$ is a normal subgroup of $\Gamma$ and by quotienting one obtains a $G$-cover $W/H\to T$. Since $H_1,\ldots H_m \le H$ the pullback of $W/H$ on $\tilde T^o$ is $\tilde V_{T^o}$. Also pullback over $\hat T_{\tau_j}$ of $W/H$ is a trivial cover for $1\le j \le N$. Hence by Theorem \ref{genpatch} the $G$-covers $W/H\to T$ and $V_T\to T$ are isomorphic. So $\Psi$ dominates $\Phi_T$.  
\end{proof}

\section{Lefschetz-Abhyankar Result} \label{LATrick}

Now we can use a Lefschetz type principle and Abhyankar's lemma to obtain a 
$\Gamma$-cover of $X$ \'etale over $C$ and dominating the given $G$-cover.  This is a deformation argument similar to \cite[Proposition 4]{HS}. 

\begin{pro}\label{ALtrick}
In the context of Proposition \ref{patchKumar}, there exist a $\Gamma$-cover $W_s\to X$ dominating the $G$-cover $\Phi:V_X\to X$ which is \'etale over $C$.
\end{pro}

\begin{proof}
 From the conclusion of Proposition \ref{patchKumar}, there is a $\Gamma$-cover $\Psi:W\to T$ of $\hat S_{s'}$-curves dominating $\Phi_T:V_T\to T$ where $V_T=V_S\times_S \hat S_{s'}$. As in the proof of Proposition \ref{patchKumar}, let $t$ be local coordinate of $S$ at $s'$ so that $\hat S_{s'}=\spec(k[[t]])$. By hypothesis $T=X_S\times_S \hat S_{s'}$, $V_S\to X_S$ is a $G$-cover and its fiber at $s\in S$ is the $G$-cover $\Phi:V_X\to X$. Also $\Psi$ is \'etale away from $B\times_{X_S} T$. Since $\Psi$ is a finite morphism, there exist a finitely generated $\cO_S$-algebra $R\subset k[[t]]$ such that the morphisms $W\to V_T\to T$ descend to the morphisms of $S_0=\spec(R)$-curves $W_{S_0}\to V_{S_0} \to T_{S_0}$. Note that the morphism $\hat S_{s'} \to S$ is the composition of the structure morphism $\pi:S_0\to S$ and the morphism $\hat S_{s'}\to S_0$ induced by the inclusion $R\subset k[[t]]$. So $T_{S_0}=X_S\times_S S_0$ and $V_{S_0}=V_S\times_S S_0$, since $T$ and $V_T$ were base change of $X_S$ and $V_S$ respectively to $\hat S_{s'}$. By shrinking $S_0$ we may assume that $S_0$ is smooth and for every point $\tau \in S_0$ the fiber of the cover $W_{S_0}\to T_{S_0}$ is a smooth irreducible $\Gamma$-cover $W_{\tau}\to X_{\pi(\tau)}$ which is \'etale away from $\{p_1(\pi(\tau)),\ldots,p_r(\pi(\tau))\}$ and dominates the $G$-cover $V_{\tau}\to X_{\pi(\tau)}$.

 Since $S_0$ is an affine $S$-scheme, we choose an embedding $S_0\into \Aff^n_S$, define $\bar S_0$ to be the closure of $S_0$ in $\PP^n_S$ and let $\bar \pi: \bar S_0\to S$ be the structure morphism. Since $\bar \pi$ is dominating and projective, it is surjective onto $S$. Let $T_{\bar S_0}=X_S\times_S \bar S_0$ and $V_{\bar S_0}=V_S\times_S \bar S_0$. Note that $T_{\bar S_0}\to \bar S_0$ and $V_{\bar S_0}\to \bar S_0$ extend $T_{S_0}\to S_0$ and $V_{S_0}\to S_0$ respectively.

 Let $W_{\bar S_0}$ be the normalization of $V_{\bar S_0}$ in $k(W_{S_0})$. Note that $W_{S_0}\to V_{S_0} \to T_{S_0}$ are finite morphism of normal varieties, so it is the restriction of $W_{\bar S_0}\to V_{\bar S_0}\to T_{\bar S_0}$. We summarize the setup in the following diagram.
 \[
  \xymatrix{
  & W\ar[r]\ar[d] & W_{S_0}\ar[r]\ar[d] & W_{\bar S_0}\ar[d]\\
    V_S\ar[d] & V_T\ar[l]\ar[r]\ar[d] & V_{S_0}\ar[r]\ar[d] & V_{\bar S_0}\ar[d]\\
    X_S\ar[d] & T\ar[l]\ar[r]\ar[d] & T_{S_0}\ar[r]\ar[d] & T_{\bar S_0}\ar[d]\\
    S & \hat S_{s'}\ar[l]\ar[r] & S_0\ar[r] & \bar S_0\ar @/^1pc/[lll]
  }
 \]

 Let $\tau_S\in \bar S_0$ be such that $\bar \pi(\tau_S)=s$. Let $\bar A_0$  be a curve in $\bar S_0$ passing through $\tau_S$ such that $A_0=\bar A_0\cap S_0$ is non empty. Note that if $\tau_S\in S_0$ then the result follows from the assumptions on $S_0$.

 Replacing $\bar A_0$ by an open neighbourhood of $\tau_S$, we may assume that the fiber at all points of $\bar A_0$ of the morphism $V_{\bar A_0}\to \bar A_0$ are smooth irreducible curves. Let $A$ and $\bar A$ be the normalization of $A_0$ and $\bar A_0$ respectively. Let $\tau_A \in \bar A$ be a point lying above $\tau_S\in \bar A_0$. 

 Let $B_1$ be the finitely many points in $\bar A \setminus A$. Let $b_i:\bar A \to V_{\bar A}$ be the section obtained by the pull-back of $p_i:S\to V_S$ along the composition of the morphisms $\bar A \to \bar A_0 \to \bar S_0 \to S$. Let $B_2=\cup_{i=1}^r \image(b_i)$ and $B_V$ be the union of fibers $\cup_{\zeta \in B_1}V_{\zeta}$. Note that $W_{\bar A}\to V_{\bar A}$ is an $H$-cover \'etale away from $B_2\cup B_V$. Since $H$ is a prime-to-$p$ group, the least common multiple $m$ of the ramification indices at the generic points of $B_V$ is coprime to $p$. Let $A'\to \bar A$ be a cyclic branched cover totally ramified at the points in $B_1$ with ramification indices $m$. Let $W_{A'}\to V_{A'}$ be the pull back of $W_{\bar A}\to V_{\bar A}$. Then applying Abhyankar's lemma we conclude that $W_{A'}\to V_{A'}$ is \'etale away from $B_2 \cup B_V$ and it is unramified at the generic points of $B_V$. Since $W_{A'}\to V_{A'}$ is a finite morphism of normal varieties, the purity of Branch locus implies that the morphism is \'etale away from $B_2$. 

 Let $\tau_{A'}\in A'$ be a point lying above $\tau_A$. The fiber over $\tau_{A'}$ of covering $V_{A'} \to T_{A'}$ is $V_s\to T_s$. But this is same as $V_X\to X$. Since $A'\to\bar A$ is proper and all the fibers of $V_{\bar A_0}\to \bar A_0$ are smooth irreducible curves, same is true for the fibers of $V_{A'}\to A'$.

 Let $b_i':A'\to V_{A'}$ be the sections obtained by the pullback of $b_i$ along $A'\to A$. After shrinking $A'$ to an open neighborhood of $\tau_{A'}$ if necessary, we may assume that $\image(b_i)$ and $\image(b_j)$ are disjoint for $i\ne j$. In particular, the branch locus of the cover $W_{A'}\to V_{A'}$ is smooth. Moreover, being a prime-to-$p$ cover, it is \'etale locally a Kummer cover. Hence the fiber over every point $\tau\in A'$ of the cover $W_{A'}\to V_{A'}$ is a cover of smooth curves. Since the fibers of $W_{A'}\to A'$ are connected for all but finitely many points of $A'$, Zariski's connectedness theorem tells us that every fiber of $W_{A'}\to A'$ must be connected. In particular, $W_{\tau_{A'}}\to T_{\tau_{A'}}$ is a $\Gamma$ cover of smooth connected curves. This cover dominates $V_{\tau_{A'}}\to T_{\tau_{A'}}$ which is same as $V_X\to X$.
\end{proof}

\section{Degenerations of covers and solving embedding problems}\label{degenerates}

In this section we use Proposition \ref{patchKumar} and \ref{ALtrick} to solve certain embedding problems. This is used to obtain more examples of effective subgroups of the fundamental group of a curve for a given embedding problem. The method below combines the technique of ``adding branch points'' as in \cite{HS} and ``increasing the genus'' as in \cite{kum-joa,kum,BK}.  

\begin{thm}\label{EPsoln_eff_degeneration}
Let $C$ be a smooth affine curve over $k$ and $X$ be the smooth completion of 
$C$. Let $g$ be the genus of $X$ and $r=\# (X\setminus C)$.
Let $\cE$\eqref{EP_degen} denote the embedding problem
\begin{equation}\label{EP_degen}
  \xymatrix{
    &          &               &\pi_1(C) \ar[d]^{\alpha} \ar @{-->} [dl]\\
    1\ar[r] & H \ar[r] & \Gamma \ar[r]^{\phi} & G \ar[r]\ar[d] & 1\\
    & & & 1 .}
\end{equation}
Let $\Phi: V_X\to X$ be the $G$-cover of $X$ \'etale over $C$ corresponding to $\alpha$. 
Let $B_X = X-C$ and $r=\#(B_X)$ and consider the set $B_X$ to be $r$ marked points on $X$.
Suppose $\Phi$ has a degeneration $\Phi':V_X'\to X'$ with associated data $(S; \Phi_S:V_S \to X_S; p_1, \ldots, p_r; s; s')$ and let  $X_1,\ldots, X_m$ be the trivial components of $X'$.
Let $r_i$ be the number of smooth marked points of $X'$ lying on $X_i$   for $i=1,\ldots,m$.
Suppose there exist a group homomorphism $$\theta:\pi_1(X_1\setminus \{r_1\text{ points}\})\times\ldots\times\pi_1(X_m\setminus \{r_m\text{ points}\})
 \to H/p(H)$$
 such that the image of $\theta$ is a relative generating set for $H/p(H)$ in $\Gamma/p(H)$.
 Then there exist a $\Gamma$-cover $W_X\to X$ which corresponds to a proper solution to $\cE$\eqref{EP_degen} (i.e. $W_X \to X$ dominates $\Phi: V_X \to X$ and is \'etale over $C$).
\end{thm}

\begin{proof} 
 Let $\bar{\Gamma}=\Gamma/p(H)$,  $\bar{H} = H/p(H)$, and  $H_i=\theta(\pi_1(X_i\setminus r_i \text{ points}))$ for $i=1,\ldots m$. 
 Consider the embedding problem $\cE$\eqref{EP_bargamma}
 \begin{equation}\label{EP_bargamma}
  \xymatrix{
    &          &               &\pi_1(C) \ar[d]^{\alpha} \ar @{-->} [dl]\\
    1\ar[r] & \bar{H} \ar[r] & \bar{\Gamma} \ar[r]^{\phi} & G \ar[r]\ar[d] & 1\\
    & & & 1 .}
\end{equation}
This has prime to $p$ kernel $\bar H$.
Since the image of $\theta$ is a relative generating set for $\bar H$ in $\bar \Gamma$, the subgroups $H_1,\ldots H_m$ and $G$ together generate $\bar{\Gamma}$. Moreover, by definition of $H_i$, there exist an $H_i$-cover of $X_i$ \'etale away from $r_i$ points for $1\le i \le m$. So the hypotheses of Proposition \ref{patchKumar} and \ref{ALtrick} for the embedding problem $\cE$\eqref{EP_bargamma} are satisfied.  By the conclusion of Proposition \ref{ALtrick} there exist a $\Gamma$-cover $W_s\to X$ dominating the $G$-cover $\Phi:V_X\to X$ which is \'etale over $C$.  This is a solution to $\cE$\eqref{EP_bargamma}.
 
 Since $p(H)$ is a quasi-$p$ group, by  \cite[Corollary 4.6]{hacrelle} or \cite[Theorem B]{pop}, the following embedding problem 
has a solution.
 \[ 
 \xymatrix{
    &          &               &\pi_1(C) \ar[d] \ar @{-->} [dl]\\
    1\ar[r] & p(H) \ar[r] & \Gamma \ar[r]^{\phi} & \bar{\Gamma} \ar[r]\ar[d] & 1\\
    & & & 1 }
 \]
This provides the required $\Gamma$-cover.
\end{proof}

As a consequence of Theorem \ref{EPsoln_eff_degeneration}, we obtain the following: 

\begin{cor}\label{num_rel-rank}
Let $C$ be a smooth affine curve over $k$ and $X$ be the smooth completion of 
$C$. Let $g$ be the genus of $X$ and $r=\# (X\setminus C)$.
Let $\cE$\eqref{EP_degen} denote the embedding problem above for $\pi_1(C)$ and let $\mu$ be the relative rank of $H/pH$ in $\Gamma/pH$.
Let $\Phi: V_X\to X$ be the $G$-cover of $X$ \'etale over $C$ corresponding to $\alpha$. Suppose $\Phi$ has a degeneration $\Phi':V_X'\to X'$ with a trivial component $X_1$. Let $r_1$ be the number of marked points of $X'$ lying on $X_1$, $g(X_1)$ be the 
 genus of $X_1$ and $n_{X_1}=2g(X_1)+r_1-1$. If one of the following holds: 
  \begin{enumerate}
   \item  $r_1\ge 1$ and $n_{X_1}\ge \mu$ 
   \item $r_1=0$ and $g(X_1)\ge \mu$
  \end{enumerate}
then there exist a $\Gamma$-cover of $X$ dominating $\Phi$ which is \'etale over $C$.
\end{cor}

\begin{proof}
 Let $u_1,\ldots, u_{\mu}$ be the relative generators of $H/pH$ in $\Gamma/pH$ and $H'$ be the subgroup of $H/pH$ generated by $u_1,\ldots, u_{\mu}$. Note that $H'$ is a prime-to-$p$ group generated by $\mu$ elements, so in both scenario $(1)$ and $(2)$ by \cite{SGA1} there exist an epimorphism from $\pi_1(X_1\setminus \{\text{marked points}\})$ to $H'$. Hence the corollary follows from Theorem \ref{EPsoln_eff_degeneration}.
\end{proof}

The above result restated in the terminology of effective subgroups becomes the following statement.
\begin{cor}\label{main.result-general}
Let $C$, $X$, $g$, $r$ and $\mu$ be as in Corollary \ref{num_rel-rank}.
Let $\Pi\normal \pi_1(C)$ be of finite index such that the embedding problem 
$\cE$\eqref{EP_degen} restricts to $\Pi$. Let $Z$ be the cover of $X$ \'etale over $C$ with 
$\pi_1(Z)=\Pi$ and $\Phi: V_X\times_X Z\to Z$ be the induced $G$-cover. 
Suppose $\Phi$ has a degeneration with a trivial component $X_1$ such that one of the following holds:
  \begin{enumerate}
   \item  $r_1\ge 1$ and $n_{X_1}\ge \mu$ 
   \item $r_1=0$ and $g(X_1)\ge \mu$
  \end{enumerate}
 then $\Pi$ is an effective subgroup for the embedding problem $\cE$\eqref{EP_degen}.
\end{cor}

\begin{proof}
 Note that $\Phi$ is a $G$-cover of $Z$ \'etale over the preimage of $C$ and it has degeneration with the same properties as in the hypothesis of the above corollary. So using that corollary, we obtain a $\Gamma$-cover of $Z$ dominating $\Phi$ which is \'etale over the preimage of $C$. This $\Gamma$-cover provides a solution to the embedding problem $\cE$\eqref{EP_degen} restricted to $\Pi$. Hence $\Pi$ is an effective subgroup.
\end{proof}

\section{The case of affine line}\label{Affine Line}
Let $C$ be the affine line $\Aff^1_x=\spec(k[x])$ and $\cE =(\alpha:\pi_1(C) \to G, \phi:\Gamma \to G)$ be an embedding problem for $\pi_1(C)$ with $H = \ker{\phi}$. In \cite[Theorem 1.3]{BK}, it was shown that there are infinitely many index $p$ effective subgroups of $\pi_1(\Aff^1_x)$ for the embedding problem $\cE$. 

In \cite[Theorem 5]{HS} (see Proposition \ref{relativerank}), it was shown that if the \'etale cover $D\to \Aff^1_x$ is such that the number of points above $x=\infty$ in the smooth completion of $D$ is large enough then $\pi_1(D)$ is an effective subgroup of $\pi_1(C)$ as long as the embedding problem $\cE$ restricts to $\pi_1(D)$.

In this section we will demonstrate some sufficient conditions on a subgroup of $\pi_1(\Aff^1_x)$ to be effective depending only on rank of $H$ and the cover corresponding to $\alpha$. In fact, we will show that in the collection of all $p$-cyclic \'etale covers of high enough genus of $\Aff^1_x$, every member $D\to \Aff^1_x$ leads to an effective subgroup $\pi_1(D)\normal \pi_1(\Aff^1_x)$.
 
Assume $\Pi$ is an index $p$ subgroup of $\pi_1(C)$. Let
$D\to C$ be the cover corresponding to $\Pi$, i.e., $D$ be the normalization
of $C$ in $(K^{un})^{\Pi}$.  
Since $D$ is an \'etale $p$-cyclic cover of the affine line, by Artin-Schrier theory it is given by the equation $z^p-z-f(x)$ for some non-constant polynomial $f(x)\in k[x]$.  
Let $r$ be the degree of $f(x)$. By changing $f(x)$ if necessary we may assume $r$ is prime to $p$. 
Let $Z \to X$ be the corresponding morphism between their smooth completions. 
The genus of $Z$ is $g_Z=(p-1)(r-1)/2$ and $Z$ is totally ramified at infinity. So $n_D = 2g_Z+r_D-1=2g_Z=(p-1)(r-1)$.

Let $\cE =(\alpha:\pi_1(C) \to G, \phi:\Gamma \to G)$ be an embedding problem for $\pi_1(C)$ with $H = \ker{\phi}$.
As observed in the introduction if $n_D$ is less than the rank of $\Gamma/p(\Gamma)$
then $\Pi$ can not be effective. But $n_D\ge \rank(\Gamma/p(\Gamma))$ is certainly not a sufficient condition 
for $\Pi=\pi_1(D)$ to be effective, as indicated by the following example.

\begin{ex}
 Let $C=\Aff^1$, $\Gamma$ a quasi-$p$ group, $G=\ZZ/p\ZZ$ and $H$ a nontrivial prime-to-$p$ group. Note that $\rank(\Gamma/p\Gamma)=0$. Suppose the map $\alpha:\pi_1(C)\to G$ be induced by a $p$-cyclic \'etale cover $V_C\to C$ where $V_C$ is also isomorphic to $\Aff^1$. Let $D\to C$ be any $p$-cyclic \'etale cover linearly disjoint from $V_C\to C$ and $U=D\times_C V_C$, then the embedding problem $\cE$ restricts to $\pi_1(D)\normal \pi_1(C)$ but if $n_U < \rank(H/p(H))$ then the embedding problem $\cE$ restricted to $\pi_1(D)$ has no solution. This is because the existence of a solution to the embedding problem implies that $H/p(H)$ is a quotient of $\pi_1(U)$. But this is impossible if $n_U<\rank(H/p(H))$. 

\end{ex}

Though we shall see that if $C=\Aff^1$, $D\to C$ is $p$-cyclic \'etale and $n_D \ge 2\rank(H/p(H))$ then $\pi_1(D)$ is indeed an effective subgroup for the embedding problem $\cE$ in many cases.

\begin{pro}\label{builddegen}\label{main.result-affine.line}
Let $\cE$\eqref{EPline} be the embedding problem
\begin{equation}\label{EPline}
  \xymatrix{
    &          &               &\pi_1(\Aff^1) \ar[d]^{\alpha} \ar @{-->} [dl]\\
    1\ar[r] & H \ar[r] & \Gamma \ar[r]^{\phi} & G \ar[r]\ar[d] & 1\\
    & & & 1 }
\end{equation}
Let $V \to \Aff^1_x$ be the $G$-Galois cover corresponding to $\alpha$ and $V_X\to \PP^1_x=X$ be the morphism corresponding to the smooth completion. Let $g$ be such that there is a homomorphism $\theta$ from the surface group $\Pi_g$ to $H/p(H)$ with the property that $\image(\theta)$ is a relative generating set for $H/p(H)$ in $\Gamma/p(H)$.
Let $D \to \Aff^1_x$ be an \'etale  $p$-cyclic cover such that the genus of the smooth completion $Z$ of $D$ is at least $g$. Let $V_Z$ be the normalization of $V_X\times_X Z$. Suppose that the genus of the normalization of $Z'\times_X V_X$ is same as $g(V_Z)$ for all but finitely $p$-cyclic covers $Z'\to X$ branched only at $x=\infty$ with the genus of $Z'$ same as $g(Z)$. Then $\pi_1(D)$ is an effective subgroup for the embedding problem $\cE$\eqref{EPline}.
\end{pro}

\begin{proof} 
Since $D\to \Aff^1_x$ is $p$-cyclic \'etale cover, it is given by the equation $Z^p-Z-(a_rx^r+a_{r-1}x^{r-1}+\ldots+a_1x)$ where $r$ is coprime to $p$, $a_r\ne 0$ and $a_{ip}=0$ for $0\le i\le [r/p]$ (\cite{pries}).
Let $A=k[t_1,\dots,t_r]$ where $t_{ip}=0$ for $1\le i\le [r/p]$ and
$t_j$ are indeterminates if $p$ does not divide $j$. Let $S=\spec(A)$ and $S^o=S\setminus \{t_r=0\}$. 
Let ${X_S}=X\times_k S$ and $Y_S$ be the normal cover of $\PP^1_y\times_k S$ given by $z^p-z-f(y^{-1})$ where
\begin{equation}\label{f}
  f(w)=w^r+t_{r-1}w^{r-1}+\ldots t_1w^1
\end{equation}
Note that the genus of the normalization of the fiber of $Y_S\to S$ for any point of $S$ is constant and $Y_S$ is normal. Hence every fiber of $Y_S$ over a closed point of $S$ is a smooth curve (Lemma \ref{smoothfibers}). 
Also note that the cover $Y_S\to \PP^1_y\times_k S$ is branched only
at $y=0$ in $\PP^1_y\times S$ because the discriminant is -1 on $\Aff^1_{y^{-1}}\times S$. Let $Y\to \PP^1_y$ be the fiber of $Y_S\to \PP^1_y\times_k S$ at the point $(t_1=0,\ldots,t_r=0)\in S$ and note that $y=0$ defines a unique point $\tau$ in $Y$.

Let $F$ be the locus of $t_r-xy=0$ in $\PP^1_x\times_k\PP^1_y\times_k S$ and $Y_F=Y_S\times_{\PP^1_y\times_k S}F$. Let $X_F={X_S}\times_{\PP^1_x\times_k S}F$ and $T$ be the fiber product $X_F\times_F Y_F$. Note that $X_F=F$ and $T=Y_F$. By definition, $\Psi_X:V_X\to X$ is a $G$-cover of $X$ \'etale over $\Aff^1_x$. 

\[
  \xymatrix{
  &  &  & T=Y_F\ar[ld]\ar[rd]\\
              &                            & X_F=F\ar[ld]\ar[d]\ar[rd] &  & Y_S\ar[ld]\\
    V_X\ar[d] & X_S=\PP^1_x\times S\ar[ld] & \PP^1_x\times \PP^1_y\times S\ar[r]\ar[l] & \PP^1_y\times S\\
    X=\PP^1_x 
  }
 \]

The morphism $T\to X_F$ is a family of covers parametrized by $S$. Let $s'\in S$ be the point $(t_1=0,\ldots,t_r=0)$ and $s$ be any point in $S^o$. Note that the morphism $F\to \PP^1_y\times S$ is an isomorphism away from $t_r=0$. So the fiber $T_s=Y_{S,s}$ is smooth. Let $T_{s'}\to X_{s'}$ be the fiber of $T\to X_F$ over $s'$. Then $X_{s'}$ is the union of $X=\PP^1_x$ and $\PP^1_y$ intersecting transversally at $(x=0,y=0)$ and $T_{s'}$ is the union $X=\PP^1_x$ and $Y$ intersecting transversally at the point $(x=0,\tau)$. The fiber over $s$, $T_s\to X_s$ is a $p$-cyclic cover of smooth curves. Since at $s\in S^o$, $t_r\ne 0$, the projection map $X_F\to X$ restricted to $X_s$ is an isomorphism. Hence $X_s=\PP^1_x=X$. Moreover, if $s$ is the point $(t_1=b_1,\ldots, t_r=b_r)$, then the cover $T_s$ is locally given by the equation $$Z^p-Z-(\frac{1}{b_r^r}x^r+\frac{b_{r-1}}{b_r^{r-1}}x^{r-1}+\ldots +\frac{b_1}{b_r}x)$$ because $y^{-1}=x{t_r}^{-1}$ on $T$.

Let $\Phi_T:V_T\to T$ be the normalized pullback of $V_X\to X$ to $T$ along the morphism $T\to X_F\to X_S\to X$. So $V_T\to T$ is a $G$-cover. Let 
$$S^1=\{s\in S^o: \text{ the genus of the normalization of }V_X \times_X T_s= g(V_Z) \}.$$ 
By hypothesis, $S^1$ is an open dense subset of $S^o$. The normalization of the fiber $V_{T,s}$ of $V_T$ at $s\in S^1$ is the normalization of $V_X\times_X T_s$. The genus of the normalization of $V_{T,s}$ for all $s\in S^1$ is constant. Moreover $V_T$ is normal, hence $V_{T,s}$ is smooth for all $s\in S^1$ (Lemma \ref{smoothfibers}). Hence $\Phi_s$ is a cover of smooth irreducible curves dominating $V_X$ for all $s\in S^1$. Let the fiber of $\Phi_T$ at $s'$ be denoted by the morphism $V_{s'}\to T_{s'}$. Since $V_X\to X$ is \'etale at $x=0$ and $T_{s'}$ is the union of $Y$ and $X$ intersecting only at $x=0$ in $X$ and $\tau$ in $Y$, the fiber $V_{s'}$ is the union of $V_X$ and $|G|$ copies of $Y$  which intersect in $V_X$ at the $|G|$ preimages of $x=0$ and in $Y$ at $y=0$. In particular for any $s\in S^1$, $\Phi_{s'}$ is a degeneration of $\Phi_s$ and the irreducible component $Y$ of $T_{s'}$ is a trivial component of $\Phi_{s'}$.

Finally choosing $b_i$, and hence $s\in S^o$, appropriately we may assume $T_s\to \PP^1_x$ is the same cover as $Z\to \PP^1_x$. More precisely, let $b_r=(a_r)^{-1/r}$, $b_i=a_ib_r^i$ for $1\le i \le r-1$ then the local equation of $T_s$ is same as that of $D$. Hence $k(T_s)=k(D)$ and the genus $g(T_s)=g(Y)$ is at least $g$. Also note that $s\in S^1$.

Note that by \cite{SGA1} there exist an epimorphism from $\pi_1(Y)$ to the prime to $p$ part of $\Pi_g$. Composing this with $\theta$ and noting that $H/p(H)$ is a prime to $p$ group, we obtain a homomorphism $\tilde\theta:\pi_1(Y)\to H/p(H)$ such that $\image(\tilde \theta)$ is a relative generating set for $H/p(H)$ in $\Gamma/p(H)$. So applying Theorem \ref{EPsoln_eff_degeneration}, we obtain a $\Gamma$-cover of $T_s$ which dominates the $G$-cover $V_X\times_XT_s\to T_s$. Hence $\pi_1(D)$ is an effective subgroup of $\pi_1(\Aff^1_x)$ for the embedding problem $\cE$\eqref{EPline}.

\end{proof}

For a Galois cover $U\to V$ of smooth connected $k$-curves, by the upper jumps at a branch point $v\in V$ we mean the upper jumps of the ramification filtration of the local field extension $\hat \cO_{U,u}/ \hat \cO_{V,v}$ for any point $u\in U$ lying above $v$. Note that the set upper jumps does not depend upon the choice of $u\in U$ lying above $v\in V$ since $U\to V$ is a Galois cover.

\begin{cor}\label{cor1:main.result-affine.line}
 Let $\cE$\eqref{EPline} be the embedding problem in Proposition \ref{main.result-affine.line}, $V_X \to \PP^1_x$ be the $G$-Galois cover corresponding to $\alpha$ and $g$ be as in Proposition \ref{main.result-affine.line}. Let $D \to \Aff^1_x$ be an \'etale  $p$-cyclic cover such that the genus of the smooth completion $Z$ of $D$ is at least $g$ and the upper jump of the cover $Z\to \PP^1_x$ at $x=\infty$ is different from all the upper jumps of $V_X\to \PP^1_x$ at $x=\infty$. Then $\pi_1(D)$ is an effective subgroup for the embedding problem $\cE$\eqref{EPline}.
\end{cor}

\begin{proof}
 Note that $Z\to \PP^1_x$ being a $p$-cyclic cover is totally ramified at $x=\infty$. Let $\tau\in Z$ be the point lying above $x=\infty$. Let $\beta_1,\ldots,\beta_l\in V_X$ be the points lying above $x=\infty$ under the morphism $V_X\to \PP^1_x$. Let $I_j\le G$ be the inertia group of $V_X\to \PP^1_x$ at $\beta_j$ for $1\le j\le l$. Since $V_X\to \PP^1_x$ is a Galois cover all the $I_j$'s are conjugates of each other.  Also we know that the degree of the morphism $V_X\to \PP^1_x$ is $|G|$ and the ramification index at $\beta_j$ is $|I_j|=e$ (say). So $|G|=el$. 
 
 Let $\hat R$ be the completion of the stalk of $V_X$ at $\beta_1$, $K$ the fraction field of $\hat R$, $\hat S$ be the completion of the stalk of $Z$ at $\tau$ and $L$ the fraction field of $\hat S$. Then $\Gal(K/k((x^{-1})))=I_1$, $\Gal(L/k((x^{-1})))=\Z/p\Z$. Since the upper jumps of the two local extensions are distinct, $K$ and $L$ are linearly disjoint over $k((x^{-1}))$. Hence $k(V_X)$ and $k(Z)$ are linearly disjoint over $k(x)$. So $V_X\times_{\PP^1_x}Z$ is an irreducible curve with the function field being the compositum $k(V_X)k(Z)$. Let $V_Z$ be the normalization of $V_X\times_{\PP^1_x} Z$. Then $V_Z\to \PP^1_x$ is a $G\times \Z/p\Z$-Galois cover branched only at $\infty$. At a point in $V_Z$ lying above $(\beta_1,\tau)\in V_X\times_{\PP^1_x}Z$, the ramification index of the cover $V_Z\to \PP^1_x$, which can be computed by passing to the completion of stalks at $\beta_1$ of $V_X$ and $\tau$ of $Z$, comes out to be the degree of the field extension $[KL:k((x^{-1}))]$. But this degree equals $pe$, since $K$ and $L$ are linearly disjoint, $[K:k((x^{-1}))]=e$ and $[L:k((x^{-1}))]=p$. Also the degree of the morphism $V_Z\to \PP^1_x$ is $p|G|=pel$. Since $V_Z\to \PP^1_x$ is a Galois cover, there are exactly $l$ points in $V_Z$ over $x=\infty$ one each lying above $(\beta_j,\tau)\in V_X\times_{\PP^1_x} Z$.
 
 Let $u_1,\ldots, u_a$ be the upper jumps of the ramification filtration on the inertia group $I_1$ of the cover $V_X\to \PP^1_x$ and $r$ be the upper jump of the inertia group $\Z/p\Z$ at $\tau$ of the cover $Z\to \PP^1_x$. Since $r$ is different from $u_1,\ldots, u_a$ by \cite[Corollary 2.5]{ramfiltr} the ramification filtration is completely determined by the ramification filtration on $I_1$ and $r$. Let $Z'\to \PP^1_x$ be another $p$-cyclic cover branched only at $x=\infty$ and the genus of $Z'$ is $g(Z)$. Then the upper jump at $x=\infty$ of $Z'\to \PP^1_x$ is also $r$ (since $g(Z)=(p-1)(r-1)/2$ depends only on $r$). Let $V_{Z'}$ be the normalization of $V_X\times_{\PP^1_x} Z'$. We observe that like $V_Z\to \PP^1_x$, the cover $V_Z'\to \PP^1_x$ is branched only at $x=\infty$, like in $V_Z$, there are exactly $l$ points in $V_{Z'}$ lying above $x=\infty$ and the ramification filtration at these $l$ points are same as the ramification filtration on the $l$ points in $V_Z$ lying above $x=\infty$. Since the degree and the ramification behaviour of $V_Z\to \PP^1_x$ and $V_{Z'}\to \PP^1_x$ are same, by Riemann-Hurwitz formula and Hilbert's different formula, the genus $g(V_{Z})=g(V_{Z'})$. The result now follows from Proposition \ref{main.result-affine.line}.
\end{proof}

\begin{cor}\label{cor2:main.result-affine.line}
Let $\cE$\eqref{EPline} be the embedding problem and $V \to \Aff^1_x$ be the $G$-Galois cover corresponding to $\alpha$ as in Proposition \ref{main.result-affine.line}.
Let $Z \to \PP^1_x$ be a $p$-cyclic cover branched only at $x=\infty$ such that $g=g(Z)$ is at least the relative rank of $H/p(H)$ is $\Gamma/p(H)$ and the upper jump of the cover $Z\to \PP^1_x$ at $x=\infty$ is different from all the upper jumps of $V_X\to \PP^1_x$ at $x=\infty$. Let $D\subset Z$ be the complement of points lying above $x=\infty$. 
Then $\pi_1(D)$ is an effective subgroup for the embedding problem $\cE$\eqref{EPline}.
\end{cor}

\begin{proof}
 Let $\{a_1,\ldots, a_g\} \subset H/p(H)$ be a relative generating set for $H/p(H)$ in $\Gamma/p(H)$. Note that $\Pi_g$ is the quotient of the free group on $2g$ generators $A_1,\ldots, A_g$, $B_1,\ldots, B_g$ by the subgroup generated by $[A_1,B_1]\cdot[A_2,B_2]\cdots[Ag,B_g]$. So there exist a homomorphism from $\Pi_g$ to $H/pH$ which takes $A_i\to a_i$ and $B_i$ to identity. Hence there exist $\theta:\Pi_g\to H/p(H)$ such that image of $\theta$ is a relative generating set for $H/p(H)$ in $\Gamma/p(H)$. The result now follows from the above corollary.
\end{proof}

\begin{rmk}\label{rmk:main.result-affine.line}
 An \'etale $p$-cyclic cover $D\to \Aff^1_x$ is given by the polynomial equation $z^p-z=f(x)$ where $f(x)$ is polynomial of degree $r$ coprime to $p$. For such a cover the upper jump of the inertia group $\Z/p\Z$ at $x=\infty$ is $r$ and genus of the smooth completion $Z$ of $D$ is $(p-1)(r-1)/2$. So given an \'etale cover $V\to \Aff^1_x$, the hypothesis on the cover $Z\to \PP^1_x$ of the above two corollaries will hold for all but finitely many values of $r$.
\end{rmk}

\section{Existence of degenerations}\label{existdegens}

We will give a few more examples below where degenerations exist.

Let $X$ be a smooth irreducible projective curve over $k$ with a nonempty set of marked points.
Let $C=X\setminus\{\text{the marked points}\}$ and $\Theta:X \to \PP^1_x$ be a finite surjective generically separable morphism 
such that $\Theta^{-1}(x=\infty)$ is any given nonempty subset of $X\setminus C$ and $\Theta$ is \'etale over $x=0$. Such a $\Theta$ exist by Lemma \ref{Noether}.
Let $m$ be the degree of the morphism $\Theta$ and $\{r_1,\ldots, r_m\}=\Theta^{-1}(x=0)$. Let $\Phi:V_X\to X$ be a $G$-cover \'etale over $C$. For a smooth $k$-variety $B$, let $\Psi_B:\scrY\to \PP^1_y\times B$ be a family of smooth covers of $\PP^1_y$ ramified only at $y=0$ and let $m'$ be the number of points in each fiber $\Psi_b:\scrY_b\to \PP^1_y$ lying above $y=0$, for all $b\in B$. Let $S=\Aff^1_t\times B$ and let $F$ be the
closure of the zero locus of $t-xy$ in $\PP^1_x\times\PP^1_y\times S$. Let $X_F=X\times_{\PP^1_x}F$,
$V_{X_F}=V_X\times_X X_F=V_X\times_{\PP^1_x} F$, $\scrY_S=\scrY\times \Aff^1_t=\scrY\times_{\PP^1_y\times B} (\PP^1_y\times S)$, $\scrY_F=\scrY_S\times_{\PP^1_y \times S} F=\scrY\times_{\PP^1_y\times B} F$, $T$ be the normalization of $X_F\times_F \scrY_F$ and $V_T$ be the normalization of $V_{X_F}\times_{X_F}T$ which is same as the normalization of $V_{X_F}\times_F \scrY_F$. Observe that the normalized base change of $\Phi:V_X\to X$ via the morphism $T\to X_F \to X$, is a $G$-cover of $S$-curves $\Phi_T:V_T\to T$.

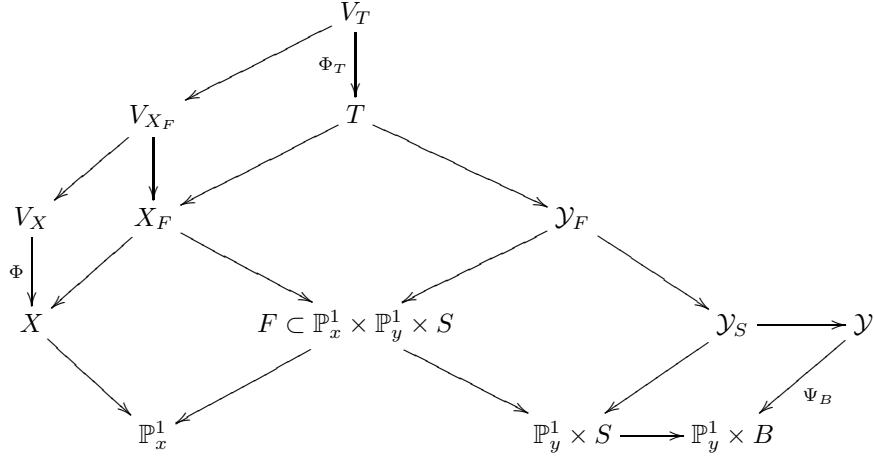
\[
  \xymatrix{
  &                & V_T \ar[ld]\ar[d]_{\Phi_T}\\
  & V_{X_F}\ar[ld]\ar[d] &  T\ar[ld]\ar[rd]\\
  V_X\ar[d]_{\Phi}  & X_F\ar[ld]\ar[rd] &  & \scrY_F\ar[ld]\ar[rd]\\
    X\ar[rd]& & F\subset \PP^1_x\times \PP^1_y\times S \ar[rd]\ar[ld]& & \scrY_S\ar[r]\ar[ld] & \scrY\ar[ld]^{\Psi_B}\\
    & \PP^1_x &  & \PP^1_y\times S\ar[r] & \PP^1_y\times B
  }
 \]

Let $\Phi':V'\to T'$ be the $G$-cover of $B$-curves obtained by looking at the fiber of $\Phi_T:V_T\to T$ over $t=0$. For $b\in B$, let $\Phi'_b:V_{0,b}\to T_{0,b}$ denote the fiber of $\Phi'$ at $b$. 

\begin{pro}\label{degeneration.highergenus}
Let the setup be as above.
 For a closed point $s=(a,b)\in S=\Aff^1_t\times B$, let $\Phi_s$ be the fiber of $\Phi_T$ at $s$. The set 
 $$U=\{s\in S: \Phi_s \text{ is a $G$-cover of smooth irreducible curves} \}$$ 
 is a nonempty open subset of $S$. Moreover for $s=(a,b)\in U$, $\Phi'_b$ is a degeneration of $\Phi_s$ and the fiber of $T$ at $(t=0,b)$ $T_{0,b}$ consist of $m$ copies of $\scrY_b$ which are trivial components of $\Phi'_b$. Also the cover $T_s\to X$ is \'etale over $C$.
\end{pro}

\begin{proof}
 The fiber $F_{0,b}$ of $F\to S$ at $(t=0,b)\in S$ for any $b\in B$ is $\PP^1_x$ and $\PP^1_y$ intersecting transversally at $x=y=0$. Let $\Psi_b^{-1}(y=0)=\{s_1,\ldots,s_{m'}\}$.  
 The fiber $T_{0,b}$ of $T\to S$ at $(0,b)$ consist of $m$ copies of $\scrY_b$ and $m'$ copies of $X$ where each copy of $X$ intersect each copy of $\scrY_b$ at exactly one point. This can be seen as follows. Let $\bar T=X_F\times_F\scrY_F$ and $T\to \bar T$ be the normalization morphism.
 \begin{align*}
  \bar T_{0,b}&=\bar T\times_S (0,b)\\
	 &=(X_F\times_F \scrY_F)\times_S(0,b)\\
	 &=(X\times_{\PP^1_x}F)\times_F(F\times_{\PP^1_y\times S}\scrY_S)\times_S(0,b)\\
	 &=(X\times_{\PP^1_x}F)\times_F(F_{0,b}\times_{\PP^1_y\times (0,b)}\scrY_b)\\
	 &=X\times_{\PP^1_x}(F\times_F ((\PP^1_x\cup\PP^1_y \text{ intersecting transversally at } x=y=0) \times_{\PP^1_y}\scrY_b))\\
	 &=X\times_{\PP^1_x}(\scrY_b \cup m' \text{ copies of $\PP^1_x$  where $j^{\text{th}}$ copy of $\PP^1_x$ intersect $\scrY_b$ at exactly}\\
	 &\text{one point } (x=0,s_j) \in \PP^1_x \times \scrY_b, 1\le j \le m') \\
	 &=\text{ $m$ copies of $\scrY_b$ and $m'$ copies of $X$ where the $j^{\text{th}}$ copy of $X$ intersect the }\\
	 &i^{\text{th}} \text{ copy of $\scrY_b$ at exactly one point } (r_i,s_j)\in X\times \scrY_b, 1\le i\le m, 1\le j \le m'.
 \end{align*}
 
  Note that $\cO_F=(\cO_{\PP^1_x}\otimes_k\cO_{\PP^1_y}\otimes_kk[t]\otimes_k \cO_B)/(t-xy)$. So $\cO_{\bar T}=(\cO_X\otimes_k\cO_{\scrY}\otimes_kk[t])/(t-xy)$ since $\cO_X$ is a flat $\cO_{\PP^1_x}$-algebra and $\cO_{\scrY}$ is a flat $\cO_{\PP^1_y}\otimes_k\cO_B$-algebra. So $\bar T$ is regular at any closed point in $t=0$ locus. Hence the normalization morphism $T\to \bar T$ is an isomorphism at these points. So $T_{0,b}=\bar T_{0,b}$.

 Note that $\Phi_T$ is the normalized base change of $V_X\to X$ and the fiber of $T\to S$ over $(t=0,b)$ consist of components isomorphic to $X$ and $\scrY_b$. So for the $G$-cover $\Phi'_b:V_{0,b}\to T_{0,b}$ each copy of $\scrY_b$ is a trivial component of $\Phi'_b$. 

 Let $\eta=S\times_{\Aff^1_t}\spec(k(t))=B\times \spec(k(t))$. For any closed point $b\in B$, let $\eta_b=b\times \spec(k(t))$ denote the corresponding closed point of $\eta$. 
 \begin{clm}
  Over the point $\eta_b$ of $S$, $V_{T,{\eta_b}}\to T_{\eta_b}$ is a  $G$-cover of irreducible $\eta_b$-curves.  
 \end{clm}
 \begin{proof}
  Note that
  \begin{align*}
    T_{\eta_b}&=T\times_S\eta_b=(\scrY_F\times_F X_F)\times_S \eta_b\\
            &=(\scrY_b\otimes_kk(t))\times_{F\times_S \eta_b} (X\otimes_kk(t))  
  \end{align*}
  So show that $T_{\eta_b}$ is irreducible, it is enough to show that the function fields of the covers $\scrY_b\otimes_k k(t)$ and $X\otimes_k k(t)$ of $F\times_S \eta_b$ are linearly disjoint over the function field of $F\times_S\eta_b$.
  Also note that the composition of morphisms $\scrY_b\otimes_kk(t)\to F\times_S\eta_b\xrightarrow{\sim}\PP^1_y\otimes_kk(t)$ is the extension of base field of $\scrY_b\to \PP^1_y$ to $k(t)$ and similarly $X\otimes_kk(t)\to F\times_S\eta_b\xrightarrow{\sim}\PP^1_x\otimes_kk(t)$ is the extension of base field of $X\to \PP^1_x$ to $k(t)$. Moreover the composition of isomorphism $\PP^1_y\otimes_kk(t)\to F\times_S \eta_b \to \PP^1_x\otimes_kk(t)$ which we will call $\Theta$ is given by $y\mapsto t/x$.
 
 Let $\alpha \in k(\scrY_b)$ be such that $k(\scrY_b)=k(y)[\alpha]$ and $f(y,Z)$ be the minimal polynomial of $\alpha$ in $k(y)[Z]$. We have a cover $\scrY_b\otimes_kk(t)\to \PP^1_x\otimes_kk(t)$ via the isomorphism $\Theta$. Consider the resulting field extension $L_2=k(t)(\scrY_b)$ of $k(t,x)$. We also have a field extension $L_1=k(t)(X)$ of $k(t,x)$ obtained from the morphism $X\to \PP^1_x$ base changed to $k(t)$. To see that $T_{\eta_b}$ is irreducible, it is enough to show that $L_1$ and $L_2$ are linearly disjoint over $k(t,x)$. Note that $L_2\cong k(t,y)[\alpha]$, hence viewing $L_2$ as an extension $k(t,x)$, we get that $L_2=k(t,x)[\alpha']$ where $\alpha'$ is a root of the irreducible polynomial $f(t/x,Z)$ in $k(t,x)[Z]$. We observe that $L_1$ and $L_2$ are linearly disjoint over $k(t,x)$ iff $[L_1L_2:L_1]=[L_2:k(t,x)]=\deg_Z(f(t/x,Z))$. But $L_1L_2=L_1[\alpha']$, so it is enough to show $f(t/x,Z)$ is irreducible in $k(t)(X)[Z]$. Let $f(t/x,Z)=Z^n+a_{n-1}(t/x)Z^{n-1}+\ldots+a_0(t/x)$ and nonzero $\gamma\in k$ be such that $x=\gamma$ is not a pole of $a_0(t/x),\ldots, a_{n-1}(t/x)$. Let $\tilde \gamma\in X$ be a closed point lying above the closed $x=\gamma$ of $\PP^1_x$. Then at point $\tilde \gamma$ the polynomial $f(t/x,Z)$ reduces to $f(t/\gamma,Z)\in k(t)[Z]$. Since $f(y,Z)$ is irreducible in $k(y)[Z]$, $f(t/\gamma,Z)$ is irreducible in $k(t)[Z]$. Hence $f(t/x,Z)$ is also irreducible in $k(t)(X)[Z]$.
 
 The proof of the irreducibility of $V_{T,\eta_b}$ is also similar and can be obtained by replacing $X$ by $V_X$ in the above argument.
 \end{proof}
 
 From the claim it follows that for any $b\in B$ there exist a nonempty open subset $U_b$ of $\Aff^1_t$ such that the fiber of $\Phi_T$ over $(a,b)\in S$ for any closed point $a$ of $U_b$ is a $G$-cover of irreducible curves. Note that $T$ and $V_T$ are normal, hence so is $T_\eta$ and $V_{T,\eta}$. So most of the fibers of $\Phi_T$ is a cover of smooth irreducible curves. Hence $U$ is a nonempty open set.

 Finally there is a morphism $T_{\eta} \to X_{\eta}=X\times_k\eta$ coming from the morphism $T\to X_F$ which is \'etale away from $y=0$. But the rational function $y$ on $X_{\eta}$ is same as $t/x$. So $T_{\eta}\to X_{\eta}$ is \'etale away from points lying over $x=\infty$. In other words, $T_{\eta}\to X_{\eta}$ is \'etale over $C\times_k \eta$.

 So for any closed point $(a,b)$ of $U\subset S=\Aff^1_t\times B$, $\Phi'_b$ is a degeneration of the $G$-cover 
 $\Phi_{a,b}:V_{a,b}\to T_{a,b}$ and $T_{a,b}\to X$ is a smooth irreducible cover \'etale over $C$ (of the same degree as $\scrY_b\to \PP^1_y$).
\end{proof}

\begin{cor} \label{cor:degeneration.highergenus}
 Let $\cE$\eqref{EPcurve} be the embedding problem 
 \begin{equation}\label{EPcurve}
  \xymatrix{
    &          &               &\pi_1(C) \ar[d]^{\alpha} \ar @{-->} [dl]\\
    1\ar[r] & H \ar[r] & \Gamma \ar[r]^{\phi} & G \ar[r]\ar[d] & 1\\
    & & & 1 }
\end{equation}
Let $X$ be the smooth completion of $C$ and $\Phi:V_X\to X$ be a $G$-cover \'etale over $C$ corresponding to $\alpha$. Let the notation and hypothesis be as in Proposition \ref{degeneration.highergenus}. Let $s=(a,b)\in U\subset S$ be a fixed point, $D\subset T_s$ be the  preimage of $C$ under the morphism $T_s\to X$ and $g$ be the genus of $\scrY_b$. Then $\pi_1(D)$ is an effective subgroup of $\pi_1(C)$ for the given embedding problem if there exist a homomorphism $\theta:\Pi_g^m\to H/p(H)$ with $\image(\theta)$ a relative generating subset of $H/p(H)$ in $\Gamma/p(H)$. Here $\Pi_g$ is the surface group of genus $g$.
\end{cor}

\begin{proof}
 By the above proposition $\Phi_s$ has a degeneration to $\Phi'_b$ with $m$ copies of $\scrY_b$ as trivial components. Also prime to  $p$ part of $\pi_1(\scrY_b)$ is the prime to $p$ part of the profinite completion of $\Pi_g$. Hence the corollary follows from Theorem \ref{EPsoln_eff_degeneration}.
\end{proof}

\begin{cor}\label{general.case}
 Let $\Pi$ be an index $p$-normal subgroup of $\pi_1(C)$ and $D\to C$ be the corresponding \'etale cover. Consider the embedding problem $\cE$\eqref{EPcurve} in the above corollary and let $V\to C$ be the $G$-Galois cover corresponding to $\alpha$. Let $X$, $V_X$ and $Z$ be the smooth completion of $C$, $V$ and $D$ respectively. Suppose there exist a separable cover $\theta:X\to \PP^1_x$ \'etale over $x=0$ with $\theta^{-1}(x=\infty)\cap C$ empty and a $p$-cyclic cover $Y\to \PP^1_x$ branched only at $x=\infty$ such that the normalization of the cover $X\times_{\PP^1_x}Y\to X$ is same as the cover $Z\to X$. Also assume that the genus $g_Y$ is at least the relative rank of $H/p(H)$ in $\Gamma/p(H)$ and the upper jump of $Y\to \PP^1_x$ is different from all the upper jumps of $V_X\to \PP^1_x$ at all the points of $V_X$ lying above $x=\infty$. Then $\Pi$ is an effective subgroup of $\pi_1(C)$ for the embedding problem \eqref{EPcurve}. 
\end{cor}

\begin{proof}
 Since $Y\to \PP^1_x$ is a $p$-cyclic cover, it is given by an Artin-Schreier polynomial $z^p-z-f(x)$ where $f(x)$ is a polynomial of degree $r$ for some $r$ coprime to $p$. Let $f(x)=a_rx^r+a_{r-1}x^{r-1}+\ldots+a_0$ with $a_r\ne 0$. Let $B=\Aff^r$ and $\scrY\to B\times \PP^1_y$ be the cover given by $z^p-z-(y^{-r}+b_{r-1}y^{-r+1}+\ldots+b_0)$ where $b_i$'s are coordinates of $B$. Note that this is a family of $p$-cyclic covers of $\PP^1_y$ branched only at $y=0$.
 
 Let $S=B\times \Aff^1_t$, $F$, $T$, $\Phi_T:V_T\to T$, etc. be defined as in the setup before Proposition \ref{degeneration.highergenus}. Let $U=\{s\in S:\Phi_s$ is a cover of smooth curves $\}$ also be as in Proposition \ref{degeneration.highergenus}.
 
 Note that $\Phi_s$ is the fiber of $\Phi_T:V_T\to T$ where $V_T$ and $T$ are the normalization of $V_{X_F}\times_F \scrY_F$ and $X_F\times_F \scrY_F$ respectively. So for $s=(t,b)\in U$ with $b=(\beta_{r-1},\ldots, \beta_0)$, $T$ is the normalization of $X\times_{\PP^1_x} \scrY_s$ where $\scrY_s\to \PP^1_x$ is given by $$z^p-z-(t^{-r}x^r+\beta_{r-1}t^{-r+1}x^{r-1}\ldots+\beta_0).$$
 
 By the hypothesis on the upper jumps of $V_X\to \PP^1_x$ and the upper jump of $Y\to \PP^1_x$ which is same as the upper jump of $\scrY_s\to \PP^1_x$, we obtain that for every point $s=(t,b)\in S$ with $t\ne 0$ the normalization of $T_s$ and $V_{T,s}$ have constant genus (i.e. independent of $s$). Hence by Lemma \ref{smoothfibers} $U=S\setminus \{t=0\}$. So by Corollary \ref{cor:degeneration.highergenus} we conclude that $\pi_1(T_s\setminus \{\text{points lying above $x=\infty$}\})$ is an effective subgroup of $\pi_1(C)$ for all $s=(t,b)\in S$ with $t\ne 0$. By an appropriate choice of $t$ and $b$ one can arrange that the cover $T_s\to X$ is same as $Z\to X$. Hence $\pi_1(D)$ is an effective subgroup of $\pi_1(C)$.
\end{proof}


\begin{thebibliography}{99}

\bibitem[BK11]{BK} Bary-Soroker, Lior and Kumar, Manish {\it Subgroup structure 
of fundamental groups in positive characteristic} arxiv, Preprint

\bibitem[Eis]{Eis} Eisenbud, David {\it Commutative Algebra with a
    view towards algebraic geometry} Book.

\bibitem[Ha94]{ha94} Harbater, David {\it Abhyankar's conjecture on
    Galois groups over curves.}  Invent. Math., 117, pages 1-25, 1994.

\bibitem[Ha03]{hacrelle}D. Harbater, \emph{Abhyankar's conjecture and embedding problems}, J. reine und angew. Math. (Crelle J.) \textbf{559} (2003), 1-24.

\bibitem[Ha99] {ha99} Harbater, David {\it Embedding problems and adding branch points}, in ``Aspects of Galois
   Theory", London Mathematical Society Lecture Note series, {\bf 256} Cambridge University Press, pages 119-143, 1999.


\bibitem[Ha03]{ha03} Harbater, David {\it Patching and Galois
    theory. Galois groups and fundamental groups} 313-424,
  Math. Sci. Res. Inst. Publ., 41, Cambridge Univ. Press, Cambridge,
  2003.

\bibitem[HS09]{HS} Harbater, David and Stevenson, Kate {\it Title:
    Embedding problems and open subgroups}
  http://arxiv.org/abs/0912.1164, Preprint.

\bibitem[Ku08]{kum-joa} Kumar, Manish {\it Fundamental group in
    positive characteristic.}  J. Algebra 319 (2008), no. 12,
  5178--5207.

\bibitem[Ku09]{kum} Kumar, Manish {\it Fundamental group of affine
    curves in positive characteristic.} J. Algebra, Volume 399, 1 February 2014, Pages 323--342.

\bibitem[Ku014]{ramfiltr} Kumar, Manish {\it Compositum of wildly ramified extensions.} J. Pure Appl. Algebra 218 (2014), no. 8, 1528-1536.

\bibitem[Pop]{pop} Pop, Florian {\it Etale Galois covers of affine
    smooth curves.} Invent. Math., 120(1995), 555-578.

\bibitem[Pri]{pries} Pries, Rachel {\it Families of wildly ramified 
covers of curves.}  Amer. J. Math.  124  (2002),  no. 4, 737-768.

\bibitem[Ray]{Ray} Raynaud, Michel  {\it Rev\^etements de la droite affine
    en caract\'eristique {$p>0$} et conjecture d'{A}bhyankar.} Invent.
  Math., 116, pages 425-462, 1994.

\bibitem[SGA1]{SGA1} Grothendieck, Alexander {\it Rev\^etements
    \'etales et groupe fondamental ({SGA} 1)} Lecture Notes in Math.,
  vol 224, Springer-Verlag, New York, 1971.

\end{thebibliography}
\end{document}